%% file: compatible_splittings.tex
\documentclass[12pt]{amsart}
\usepackage{amsmath}
\usepackage{amssymb}
\usepackage{amsthm}
\usepackage{amscd}
\usepackage{amsfonts}
\usepackage{graphicx}%
\usepackage{fancyhdr}
\usepackage{enumerate}
\usepackage{mathrsfs}
\usepackage[all]{xy}

\def\mf#1{{\mathfrak{#1}}} %mathfrak!
\def\mc#1{{   \mathcal{#1}   }}
\def\mb#1{{   \mathbb{#1}   }}

\def\st{{  \tr{St}  }}
\def\stg{{  \st  }}

\def\hzero#1{{  H^0( #1 )  }}
\def\canpd#1#2{{  \canp{#1}\big(  (1-p) #2  \big)  }}
\def\Endf#1{{  \mc{E}nd_F ( \struct {#1} )  }}

\def\reg{{  \tr{reg}  }}
\def\sing{{  \tr{sing}  }}

\def\endf#1{{  \tr{End}_F ( #1 )  }}
\def\pst{{  (p-1)^{st}  }}

\input{definitions1.tex}

\begin{document}

%\raisebox{1cm}

%Type here!
\title{On the $B$-canonical splittings of flag varieties}
\author{Chuck Hague}
\begin{abstract} Let $G$ be a semisimple algebraic group over an algebraically closed field of positive characteristic. In this note, we show that an irreducible closed subvariety of the flag variety of $G$ is compatibly split by the unique canonical Frobenius splitting if and only if it is a Richardson variety, i.e. an intersection of a Schubert and an opposite Schubert variety.
\end{abstract}
\maketitle
\bibliographystyle{amsplain}
\section{Introduction} Let $G$ be a semisimple simply-connected algebraic group over an algebraically closed field of positive characteristic and let $B \subseteq G$ be a Borel subgroup. By work of Mathieu \cite{Mat90} and Ramanathan \cite{Ra87}  (cf also \cite{BK}, \S2.3, \S4.1-4.2 and \cite{vdK93}), the flag variety $G/B$ admits a unique $B$-canonical splitting that compatibly splits all intersections of Schubert and opposite Schubert varieties (the so-called Richardson varieties). In this note, we show that the converse holds: Any irreducible closed subvariety of $G/B$ that is compatibly split by $\psi$ is a Richardson variety.

Here is an outline of the proof. We first show in Theorem \ref{th:comp split normal system iff} that if $\varphi$ is a splitting of a normal variety $X$ such that $\varphi$ is a $\pst$ power and such that the divisor of zeroes $Z( \varphi )$ generates a normal intersection compatible system $\mf X$ (cf Definitions \ref{def:bundles on a normal variety} and \ref{def:normal intersection compatible system} below), then a closed irreducible subvariety $Z \subseteq X$ is compatibly split by $\varphi$ if and only if $Z \in \mf X$. Furthermore, let $\psi$ denote the unique $B$-canonical splitting of $G/B$; then $\psi$ is a $\pst$ power and $Z(\psi)$ generates an intersection compatible system whose elements are precisely the Richardson varieties. It is straightforward to check (cf Theorem \ref{th:richardson facts}) that the set of Richardson varieties is a normal system. The result then follows immediately.

This note is based on a question asked by Allen Knutson. After I obtained this result, he informed me that he has independently obtained the same result. I would also like to thank Michel Brion, Shrawan Kumar, George McNinch, and Karl Schwede for helpful comments and discussions.

\section{Splitting facts}

\subsection{Review of Frobenius splitting methods}
\quad

In this section we review the theory of Frobenius splitting. The main references are \cite{BK} and the seminal paper \cite{MR85}. In the sequel we assume that all varieties are over an algebraically closed field $k$ of positive characteristic $p$.

Let $X$ be a scheme over $k$. We define a morphism $F_X$ of schemes over $\mb{F}_p$ as follows. Set $F_X(x) = x$ for all $x \in X$ and define $F_X^\# : \struct{X} \rightarrow (F_X)_* \, \struct{X}$ to be the $p^{\textrm{th}}$ power map $f \mapsto f^p$; this is clearly an $\mb{F}_p$-linear map. Note that $F_X$ is not a morphism of schemes over $\algclosed{p}$. This morphism is called the \textbf{absolute Frobenius morphism}. When the context is clear we'll drop the subscript and just write $F$.

%Note that for any sheaf $\mc{F}$ of $\struct{X}$-modules, the sheaf $F_* \mc{F}$ on $X$ is isomorphic to $\mc{F}$ as a sheaf of abelian groups, but the $\struct{X}$-module structure is twisted: for $m \in F_* \mc{F}$ and $f \in \struct{X}$ we have the twisted action $f*m = f^p \, m$.

\begin{definition} We say that $X$ is \textbf{Frobenius split} if there is an $\struct{X}$-linear map $\varphi : F_* \struct{X} \rightarrow \struct{X}$ such that $\varphi \circ F^\#$ is the identity map on $\struct{X}$.
\end{definition}

Let $ \Endf X $ denote the sheaf of $\struct X$-linear maps $ F_* \struct{X} \rightarrow \struct{X} $. Then $ \Endf X $ is a sheaf of $\struct X$-modules, where the $\struct X$-action is given by $(f.a)(h) = f \cdot (a(h))$ for $f \in \struct X$, $h \in F_* \struct X$, and $a \in \Endf X$. Set $ \endf X := \Cohom{0}{X}{ \Endf X } $.

Let $ \varphi \in \endf X $ be a splitting of $X$. We say that a closed subvariety $Y \subseteq X$ is \tb{compatibly split} by $\varphi$ if $ \varphi ( F_* \ms I ) \subseteq \ms I $, where $\ms I \subseteq \struct X$ is the ideal sheaf of $Y$. We have the following useful facts:
	\begin{enumerate}[(i)]
	\item $ \varphi \in \endf X $ is a splitting of $X$ if and only if $\varphi|_U$ it is a splitting of $U$, for any dense open subvariety $U \subseteq X$.
	\item If $Y \subseteq X$ is a closed subvariety that is compatibly split by a splitting $\varphi$ of $X$, then each irreducible component of $Y$ is also compatibly split.
	\item If $Y$ and $Z$ are closed subvarieties of $X$ that are compatibly split by a splitting $\varphi$ of $X$, then the set-theoretic intersection $Y \cap Z$ and union $Y \cup Z$ are also compatibly split.
	\end{enumerate}

\begin{definition} \label{def:bundles on a normal variety} Let $X$ be a normal variety. Let $X^\sing$ denote the singular locus of $X$ and let $X^\reg$ denote the regular locus. Recall that for all $n \in \mb Z$ we have $\omega_X^n = i_* \omega_{ X^\reg }^n$, where $i : X^\reg \hookrightarrow X$ is the inclusion. Let $s \in \Cohom{0}{X}{ \canp X }$; then we set $Z(s) := \overline{ Z( s|_{ X^\reg } ) }$, where by $Z( s|_{ X^\reg } )$ we mean the (set-theoretic) zero set of $s|_{ X^\reg }$. Also, we say that $s$ is a $\pst$ power if the restriction of $s$ to $X^\reg$ is a $\pst$ power.
\end{definition}

If $X$ is an $H$-scheme for an algebraic group $H$, then there is a natural action of $H$ on $ \endf X $ given by $$ ( h.\varphi )( f ) \, = \, h ( \varphi ( h^{-1} f ) ) \, ,$$ for all $h \in H$, $\varphi \in \endf X$, and $f \in F_* \struct X$.

\begin{theorem} \label{th:endf and canp} Let $X$ be a normal variety. Then there is a $\struct X$-module isomorphism $ \mc End_F ( \struct X ) \isom \canp X $. Furthermore, if $X$ is an $H$-variety for an algebraic group $H$, then the induced isomorphism $ \endf X \isom \Cohom{0}{X}{ \canp X } $ is compatible with the $H$-structure on $\endf X$ defined above and the natural $H$-structure on $ \Cohom{0}{X}{ \canp X } $.
\end{theorem}

Using this theorem, we will from now on consider Frobenius splittings to be elements of $\Cohom{0}{X}{ \canp X }$.

We also recall the following facts (cf \cite{BK}, Exercises 1.3.E (3) and (4)): Assume that $X$ is normal and let $D$ be a closed subvariety of pure codimension 1. Similarly to above, we set $$  \canpd X D \, := \, i_* \canpd{ X^\reg }{ (D \cap X^\reg) }  \, ,$$ where $i : X^\reg \hookrightarrow X$ is the inclusion. We have a similar definition for $ \cand X{-D} $. Let $ \varphi \in \Cohom{0}{X}{ \canp X } $ be a splitting of $X$. Then we have:
	\begin{enumerate}[(i)]
	\item $D$ is compatibly split by $\varphi$ if and only if $\varphi$ is in the image of the natural morphism $$ \Cohom{0}{X}{ \canpd X D } \hookrightarrow \Cohom{0}{X}{ \canp X } \, .$$ Further, if $\varphi$ compatibly splits $D$ then the degree of vanishing of $\varphi$ on $D$ is exactly $p-1$. 

	\item If $ \varphi $ is a $\pst$ power then $D$ is compatibly split if and only if $D \subseteq Z( \varphi )$.
	\end{enumerate}
(Although these results are stated in \cite{BK} only for the smooth case, it is easy to check that the results extend to the normal case). In particular, if $\varphi$ is a $\pst$ power and $D$ is a compatibly split divisor, then $\varphi$ is in the image of the $\pst$ power map $$ \Cohom 0 X { \cand X {-D} } \stackrel {\otimes (p-1)} \longrightarrow \Cohom{0}{X}{ \canpd X D } \hookrightarrow \Cohom{0}{X}{ \canp X } \, .$$

 We will need the following result from \cite{KM}.

\begin{proposition} \label{pr:km} (\cite{KM}, Proposition 2.1) Let $X$ be a smooth variety and let $\varphi \in \Cohom{0}{X}{ \canp X }$ be a splitting of $X$. If $ Y \subsetneq X $ is compatibly split by $\varphi$, then $Y \subseteq Z( \varphi )$.
\end{proposition}

\subsection{Three lemmas on splittings of normal varieties}
\quad

We will need the three following technical results.

\begin{lemma} \label{lem:nth power} Let $X$ be a normal irreducible variety and let $\mc L$ be a line bundle on $X$. Let $n \in \mb Z^+$ and let $s \in \cohom{0}{X}{ \mc L^n }$ be such that $s|_U$ is an $n^{th}$ power on some nonempty open set $U \subseteq X$. Then $s$ is an $n^{th}$ power.
\end{lemma}

\begin{proof} Let $t \in \cohom{0}{ U }{ \mc L|_U }$ be such that $t^n = s|_U$. Since $s$ has no poles on $X \setminus U$, neither does $t$. Hence, as $X$ is normal, $t$ extends to a global section $a \in \cohom{0}{ X }{ \mc L }$. Since $a^n|_U = s|_U$, we must have $a^n = s$ and hence $s$ is an $n^{th}$ power.
\end{proof}

\begin{lemma} \label{lem:p-1st powers on smooth varieties} Let $X$ be a normal variety. Let $\psi \in \Cohom{0}{X}{ \canp X }$ be a splitting and let $D \subseteq X$ be an compatibly split irreducible normal prime divisor. If $\psi$ is a $\pst$ power, then so is the induced splitting of $D$.
\end{lemma}

\begin{proof} Since $X$ is normal, $D^\reg \cap X^\reg \neq \O$. By Lemma \ref{lem:nth power}, it suffices to check that the restriction of $\psi$ from $X^\reg$ to $D^\reg \cap X^\reg$ is a $\pst$ power, so we may assume that $X$ and $D$ are smooth. The result now follows from the commutativity of the following diagram: $$  \xymatrix{    \Cohom{0}{ X }{ \can X (-D) } \ar[d] \ar[r]^{ \hspace{-.2in} \otimes(p-1) } & \Cohom{0}{ X }{ \canp X ( (1-p)D ) } \ar[d] \\
\Cohom{0}{ D }{ \can D } \ar[r]^{ \otimes(p-1) } & \Cohom{0}{ D }{ \canp D }   }  $$  where the horizontal arrows are the natural maps induced by the isomorphism \newline $ \can X( -D )|_D \cong \can D $.
%\item Replacing $X$ by $X^\reg$ we can assume that $X$ is smooth; but the result is clear in that case.
\end{proof}

For any $s \in \Cohom{0}{ X }{ \canpd X D }$, we let (as above) $Z(s)$ denote the closure of $Z( s|_{X^\reg} )$; note that if $\canpd X D$ is a line bundle this agrees with the usual definition of the zero set of $s$. We have a similar definition for $Z(s)$ when $s \in \Cohom{0}{X}{ \cand X {-D} }$.

Recall the following result (\cite{flips}, Proposition 16.4): If $X$ is a normal variety and $D$ is a normal subvariety of $X$ of pure codimension 1 such that $ \cand X {-D} $ is a line bundle on $X$, then $ \cand X {-D}|_D \cong \can D $. In this case, we clearly also have $ \canpd X D|_D \cong \canp D $.

\begin{lemma} \label{lem:the case where can is a line bundle} Let $X$ be a normal, irreducible variety and let $D \subseteq X$ be a normal prime divisor such that $ \cand X {-D} $ is a line bundle on $X$. Let $$ p : \COhom{0}{X}{ \canpd X D } \to \Cohom{0}{D}{ \canp D } $$ be the map induced by the isomorphism $ \canpd X D|_D \cong \canp D $ and choose $s \in \COhom{0}{X}{ \canpd X D }$. Then $Z( p(s) ) = Z( s ) \cap D$.
\end{lemma}

\begin{proof} By passing to a small open set, we may assume that $X$ is affine and that $ \cand X {-D} $ trivializes on $X$. Then $\canpd X D$ trivializes on $X$ as well. Let $\theta \in \Cohom{0}{ X }{ \canpd X D }$ be a trivialization; then $ \theta|_D $ is a trivialization of $ \Cohom{0}{D}{ \canp D } $. Thus, using $\theta$, the map $p$ identifies with a surjection $ p' : \struct X \twoheadrightarrow \struct D $. Since $p$ is just section restriction, $p'$ is the natural restriction of functions; but then the result follows from the corresponding result for functions.
\end{proof}

\subsection{Systems of subvarieties}

\begin{definition} \label{def:normal intersection compatible system} Let $X$ be a variety. Following \cite{ST09}, we say that a set $\mf X$ of closed irreducible subvarieties of $X$ is \tb{intersection compatible} if the set of finite unions of elements of $\mf X$ is closed under set-theoretic intersection. If $A$ is a set of closed subvarieties of $X$, let $ \mf S(A) $ denote the minimal intersection-compatible system containing $A$. That is, $ \mf S(A) $ is the system obtained by iteratively taking set-theoretic intersections and irreducible components.

For any intersection compatible system $ \mf X $ and any $Y \in \mf X$, set $$ \mf X^Y := \{Y\} \cup \mf S \big( \{  E \in \mf X : E \tr{ is a divisor of } Y  \} \big)  \, ,$$ an intersection compatible subsystem of $\mf X$. Also set $ \overline{ \mf X^Y } := \{  Z \in \mf X : Z \subseteq Y  \} $, another intersection compatible subsystem of $ \mf X $. Clearly $ \mf X^Y \subseteq \overline{ \mf X^Y } $.

We say that an intersection compatible system $\mf X$ is \tb{normal} if: (1) All elements of $\mf X$ are normal; (2) For all $Y \in \mf X$ we have $ \mf X^Y = \overline{ \mf X^Y } $; and (3) for all $Y \in \mf X$ such that $Y^\sing \neq \O$ and for every irreducible component $Z$ of $Y^\sing$, there is a prime divisor $D $ of $ Y$ such that $Z \subseteq D$ and $D \in \mf X$. Note that (2) implies that for every $Y \in \mf X$, $ \mf X^Y $ is a normal intersection compatible subsystem of $\mf X$.

Given a Frobenius splitting $\psi$ of a variety $X$, let $D$ be the union of all prime divisors compatibly split by $\psi$ and set $\mf X_\psi := \{X\} \cup \mf S( D )$. In this case, the set of finite unions of elements of $ \mf X_\psi $ is closed even under scheme-theoretic intersection, not just set-theoretic intersection.
\end{definition}

\begin{remark} Although it is clear that every element of $ \mf X_\psi $ is a compatibly split subvariety, I don't know if the converse holds in general; although see Theorem \ref{th:comp split normal system iff} below for a partial converse.
\end{remark}

\begin{lemma} \label{lem:restriction of primary system} Let $X$ be a normal variety and let $\psi \in \Cohom{0}{X}{ \canp X }$ be a splitting of $X$ such that $\psi$ is a $\pst$ power and $ \mf X_\psi $ is a normal intersection compatible system. Let $D \subseteq X$ be a compatibly split prime divisor and let $\varphi$ be the induced splitting of $ D $. Then we have $  \mf X_\varphi \, = \, \mf X_\psi^D .$ In particular, $\mf X_\varphi$ is also a normal system.
\end{lemma}

\begin{proof} Set $\mf Y := \{ Z \in \mf X_\psi : Z \tr{ is a divisor in } D \}$. Since $\varphi$ is a $\pst$ power, it suffices to check that $\mf Y$ is precisely the set of irreducible components of $ Z( \varphi ) $. Since $\psi$ compatibly splits $D$, it is in the image of the natural map $ \canpd X D \hookrightarrow \canp X $, and we will think of $\psi$ as a section of $\canpd X D$. Let $Z'( \psi )$ denote the zero set of $\psi$ considered as a section of $ \canpd X D $; then $ Z'( \psi ) $ is the union of all irreducible components of $ Z( \psi ) $ other than $D$.

We first check that each irreducible component of $ Z( \varphi ) $ is in $ \mf Y$. If $Z( \varphi ) = \O$ the result is trivial, so assume that $Z( \varphi ) \neq \O$. Let $E$ be an irreducible component of $Z( \varphi )$. Assume, to the contrary, that for every compatibly split prime divisor $D' \neq D$ of $X$ we have that $ E $ is not an irreducible component of $D' \cap D$. Since $\psi$ is a $\pst$ power, this implies that $E$ is not an irreducible component of $D' \cap D$ for any irreducible component $D'$ of $Z( \psi )$. In particular, $E \nsubseteq Z'(\psi)$. Replacing $X$ with $X \setminus Z'( \psi )$, we restrict to the case where $Z'( \psi ) = \O$.

Since $\psi$ is a $\pst$ power, we can write $\psi = \sigma^{p-1}$ for some $\sigma \in \Cohom 0 X {\cand X{-D}}$. Since $\psi$ has no zeroes on $X$, neither does $\sigma$. Hence $\sigma$ provides a trivialization $ \cand X {-D} \stackrel{\sim}{\to} \struct X $ on $X^\reg$. Since $ \cand X {-D} $ and $ \struct X $ are reflexive, this trivialization extends to an isomorphism on $X$; in particular, $\cand X{-D}$ is a line bundle on $X$. Thus, by Lemma \ref{lem:the case where can is a line bundle}, $Z( \varphi ) = Z'( \psi ) \cap D = \O$, which is false (note that $D$ is normal since $\mf X_\psi$ is a normal system). Hence $ E \in \mf Y $.

Conversely, to see that $ \mf Y \subseteq Z( \varphi ) $, let $E \in \mf Y$; then, as $E$ is compatibly split by $\psi$, it is also compatibly split by $\varphi$, and the result obtains since $E \cap D^\reg \neq \O$. 
\end{proof}

\begin{theorem} \label{th:comp split normal system iff} Let $X$ be a normal variety and let $\psi$ be a splitting of $X$ such that $\psi$ is a $\pst$ power and $ \mf X_\psi $ is a normal intersection compatible system. Then a closed irreducible subvariety $Z$ of $X$ is compatibly split by $\psi$ iff $Z \in \mf X_\psi$.
\end{theorem}

\begin{proof} The "if" part is clear. We prove the converse by induction on $\tr{dim} \, X$. If $ \tr{dim} \, X = 0 $ then the result is trivial; now assume that the result holds for all normal varieties of dimension $< n$ that satisfy the conditions of the theorem. Let $X$ be a normal variety of dimension $n$ and let $\psi$ be a splitting of $X$ such that $\psi$ is a $\pst$ power and $ \mf X_\psi $ is a normal system. Let $Z \subseteq X$ be compatibly split by $\psi$.

I first claim that there is $D \in \mf X_\psi$ such that $D$ is a divisor of $X$ and $ Z \subseteq D $. If $ Z \cap X^\reg \neq \O $ then $ Z \cap X^\reg \subseteq Z( \psi ) $ by Proposition \ref{pr:km}, and hence there is an irreducible component $D$ of $Z( \psi )$ such that $Z \subseteq D$. Since $\psi$ is a $\pst$ power, $D$ is compatibly split by $\psi$, so $D \in \mf X_\psi$, as desired. On the other hand, if $Z \subseteq X^\sing$, let $E$ be an irreducible component of $X^\sing$ containing $Z$. Since $\mf X_\psi$ is a normal system there is a prime divisor $D$ of $X$ such that $ Z \subseteq E \subseteq D $ and $D \in \mf X_\psi$. Hence the claim holds.

Let $\varphi$ denote the induced splitting of $D$; we now obtain an intersection compatible system $\mf X_\varphi$ on $D$. By Lemma \ref{lem:p-1st powers on smooth varieties}, $\varphi$ is a $\pst$ power, and by Lemma \ref{lem:restriction of primary system}, $\mf X_\varphi = \mf X_\psi^D$ is a normal intersection compatible system containing $Z$. The result now follows by the induction hypothesis.
\end{proof}

\section{Compatibly split subvarieties of flag varieties}

Let $G$ be a semisimple simply-connected algebraic group over $k$. Fix a maximal torus $T \subseteq G$ and a Borel subgroup $B$ of $G$ containing $T$. Recall that the \textbf{weights} of $G$ are the algebraic group homomorphisms $T \to k^*$. For any weight $\lambda$ of $G$ let $\mc L( \lambda )$ denote the equivariant bundle $G \times^B k_{ -\lambda }$ on $G/B$ with fiber $k_{-\lambda}$, and set $$  \hzero \lambda := \Cohom{0}{ G/B }{ \mc L( \lambda ) }  \, .$$ Let $\rho$ denote the half-sum of the positive roots of $G$ and set $\stg := \hzero{ (p-1) \rho }$, the \tb{Steinberg module}. 

\begin{definition} A Frobenius splitting $\psi$ of a $B$-variety $X$ is called \tb{$B$-canonical} if there is a $B$-equivariant morphism $\theta : \stg \otimes k_{ (p-1) \rho } \to \Cohom{0}{ X }{ \canp X }$ such that $\psi = \theta( f_- \otimes f_+ )$, where $f_- \in \stg$ is any nonzero lowest weight vector and $f_+ \in k_{ (p-1) \rho }$ is any nonzero vector.
\end{definition}

For any $w$ in the Weyl group $W$ of $G$ set $C_w := BwB / B$, the Schubert cell in $G/B$ corresponding to $w$ and let $X_w := \overline{ C_w }$ denote the Schubert variety corresponding to $w$. We similarly have the opposite Schubert cell $C^w := B^- w B / B$ and opposite Schubert variety $X^w := \overline{C^w}$. For any pair $v, w \in W$ set $ C_w^v := C_w \cap C^v $ and $ X_w^v := X_w \cap X^v $; these varieties are called \tb{Richardson varieties}.

The following is essentially due to Mathieu \cite{Mat90} (see the comments following Lemma 2.3); see also \S2.3 and \S4.1 in \cite{BK} and \S4.3 in \cite{vdK93}.

\begin{theorem} \label{th:B-canonical splitting facts} There is a unique $B$-canonical $\psi$ splitting of $G/B$. Moreover, $\psi$ is a $\pst$ power, and $Z(\psi)$ is the union of the codimension-1 Schubert and opposite Schubert varieties. In particular, $ \mf X_\psi $ is precisely the set of all Richardson varieties.
\end{theorem}

\begin{proof} The existence and uniqueness of $\psi$ is known (see the references above). We now check that $\psi$ is a $\pst$ power. Indeed, let $m : \hzero \rho \otimes \hzero \rho \to \hzero{2 \rho}$ be the multiplication map; then $\psi = m(v_+ \otimes v_-)^{ \otimes (p-1) }$, where $v_+$ (resp. $v_-$) is a nonzero highest (resp. lowest) weight vector in $\hzero \rho$. Finally, the facts about $Z(\psi)$ and $\mf X_\psi$ follow from the proof of Theorem 2.3.1 in \cite{BK}.
\end{proof}

Parts (i), (ii), and (iii) of the following theorem are just a restatement of Theorem 3.2 in \cite{Ri92} and Lemma 1 of \cite{BL03}.

\begin{theorem} \label{th:richardson facts} Let $v, w \in W$.
	\begin{enumerate}[(i)]
	\item $X_w^v$ is nonempty if and only if $v \leq w$, in which case $X_w^v$ is a normal irreducible variety of dimension $ l(w) - l(v) $. Furthermore, $X_{w'}^{v'} \subseteq X_w^v$ if and only if $ v \leq v' \leq w' \leq w $.
	\item $C_w^v$ is an open nonsingular subvariety of $X_w^v$.
	\item The boundary $ \partial X_w^v := X_w^v \setminus C_w^v $ of $X_w^v$ is a union of Richardson varieties.
	\end{enumerate} In particular, the set of Richardson varieties forms a normal intersection compatible system.
\end{theorem}

\begin{proof} We need to verify the "in particular" part. By (i) each Richardson variety is normal, so part (1) of the definition of a normal system is satisfied. By (ii) and (iii), the singular locus of a Richardson variety is contained in a union of Richardson varieties, and by (i) every Richardson subvariety of a Richardson variety $X^v_w$ is contained in a Richardson divisor of $X^v_w$. Thus part (3) of the definition of a normal system follows. We now need to check part (2) of the definition of a normal system.

We show, by induction on dimension, the following: Let $Y$ be a Richardson variety and let $Z$ be the union of all Richardson varieties that are divisors in $Y$. Then $ \mf S( Z ) $ is the set of all Richardson varieties contained in $Z$. This is trivial for 0-dimensional Richardson varieties, so now choose $v, w \in W$ with $v < w$ and assume that the induction hypothesis holds for all Richardson varieties of dimension $< l(w) - l(v)$. We need to check the following: \begin{itemize} \item[$(*)$] For all Richardson divisors $D$ in $ X_w^v $, and for any Richardson divisor $E$ of $D$, there is a Richardson divisor $D'$ of $X_w^v$ such that $E$ is an irreducible component of $D' \cap D$.
	\end{itemize}
Now, the Richardson divisors of $X_w^v$ are the $X_w^a$ for all $a > v$ with $l(a) = l(v) + 1$ and the $X_b^v$ for all $b < w$ with $l(b) = l(w) - 1$. We'll verify that $(*)$ holds in the case $D = X^v_b$; a similar argument (or translation by $w_0$) will then give $(*)$ for the case $D = X^a_w$.

Fix $b < w$ with $l(b) = l(w) - 1$. We first consider the divisor $ X_{b'}^v $ of $X_b^v$ for $b' < b$, $l(b') = l(b) - 1 = l(w) - 2$. By Lemma 10.3 in \cite{BGG75}, there exists $x \in W$ with $x \neq b$ and $b' < x < w$. Hence $X_{b'}$ is an irreducible component of $ X_b \cap X_x $, so that $X^v_{b'}$ is an irreducible component of $ X^v_x \cap X^v_b $. Since $ X^v_x $ is a divisor in $X^v_w$, the result now follows for $ X_{b'}^v $.

Now consider the divisor $X_b^{v'}$ of $X_b^v$ for $v' > v$, $l(v') = l(v) + 1$. Then $$ X^{v'}_w \cap X^v_b \, = \, X_w \cap X_b \cap X^{v'} \cap X^v \, = \, X_b^{v'} \, .$$ Since $ X^{v'}_w $ is a divisor in $X_w^v$, the result now follows, and the proof is complete.
\end{proof}

We now have the main result.

\begin{theorem} Let $\psi \in \Cohom{0}{ G/B }{ \canp{G/B} }$ be the unique $B$-canonical splitting of $G/B$. Then a closed subvariety of $G/B$ is compatibly split by $\psi$ if and only if it is a union of Richardson varieties.
\end{theorem}

\begin{proof} This is immediate from Theorems \ref{th:comp split normal system iff}, \ref{th:B-canonical splitting facts}, and \ref{th:richardson facts}.
\end{proof}

\newpage

\bibliography{thebibliography}
\end{document}

%% file: definitions1.tex
\theoremstyle{plain} \numberwithin{equation}{subsection}
\newtheorem{theorem}{Theorem}[section]

\newtheorem{lemma}[theorem]{Lemma}
\newtheorem{proposition}[theorem]{Proposition}
\newtheorem{proposition/definition}[theorem]{Proposition/Definition}

\theoremstyle{definition}
\newtheorem{remark}[theorem]{Remark}

\newtheorem{definition}[theorem]{Definition}

\def\mf#1{{\mathfrak{#1}}} %mathfrak!
\def\mc#1{{   \mathcal{#1}   }}
\def\mb#1{{   \mathbb{#1}   }}
\def\ms#1{{   \mathscr{#1}   }}

\def\tb#1{{\textbf{#1}}}
\def\tr#1{{\textrm{#1}}}

\def\cohom#1#2#3{{H^{#1}  ( #2, \, #3 )}} %#1-th cohomology on #2 with coeffs in #3
\def\Cohom#1#2#3{{H^{#1} \big( #2, \: #3 \big)}} %big cohom
\def\COhom#1#2#3{{H^{#1} \Big( #2, \: #3 \Big)}} %Big cohom

\def\struct#1{{\mathcal{O}_{#1}}} %structure sheaf
\def\isom{{\ \tilde{\longrightarrow} \ }} %isomorphism
\def\algclosed#1{{   \bar{  \mb{F}  }_{#1}   }}
\def\can#1{{  \omega_{ #1 }^{-1}  }}
\def\cand#1#2{{  \can{#1}( #2 )  }}
\def\canp#1{{  \omega_{#1}^{1-p}  }}